\documentclass[11pt]{amsart}
\usepackage{amsmath,amssymb,amsthm,amsfonts}

 \usepackage{epsfig} 
 \usepackage{graphics} 

\newcommand{\abs}[1]{| #1 |}
\newcommand{\Abs}[1]{\left| #1\right|}

\def\C {\mathbb C}
\newcommand{\cl} {\overline}
\newcommand{\dd}{\delta}

\newcommand{\D}{\mathbb D}
\newcommand{\e}{\epsilon}

\newcommand{\N}{\mathbb N}
\newcommand{\norm} [1]{\left\| #1\right\|}

\newcommand{\p}{\partial}
\newcommand{\R}{\mathbb R}
\newcommand{\re}{\text{\rm Re}\,}

\newcommand{\set}[1]{\left\{ #1\right\}}
\def\sm{\setminus}

\newcommand{\To}{\longrightarrow}

\newcommand{\W}{\Omega}
\newcommand{\z}{\zeta}

\newtheorem{theorem}{Theorem}

\newtheorem{prop}[theorem]{Proposition}

\theoremstyle{definition}
\newtheorem{definition}[theorem]{Definition}

\theoremstyle{remark}

\title[Kobayashi metric]
      {Asymptotic Behavior of the Kobayashi Metric in the Normal Direction}

\author{John Erik Forn\ae ss, Lina Lee}

\email{ fornaess@umich.edu, linalee@umich.edu}

\thanks{The first author is supported by an NSF grant.}

\begin{document}

\maketitle

\begin{abstract}
In this paper we construct a pseudoconvex domain in $\C^3$ where the Kobayashi metric does not blow up at a rate of one over distance to the boundary in the normal direction.
  \end{abstract}

  \section{Introduction}

  The asymptotic behaviour of invariant  metrics has been studied by several authors, Graham \cite{Graham}, Catlin \cite{Catlin}, Krantz \cite{Krantz 1992}, Diederich-Herbort \cite{dieher}, Lee \cite{Lina}, and others.

In this paper, we study the behavior of the Kobayashi metric, which is defined as follows:

\begin{definition}
Let $\W\subset\C^n$ be a domain, $Q\in\W$ and $X\in T_Q(\W)$. The Kobayashi metric $F_K:T\W\To\R^+\cup\set 0$ is defined as
$$
F_K(Q,X)=\inf\set{\alpha>0:\exists\phi\in\W(\D),\, \phi(0)=Q,\,\phi'(0)=X/\alpha},
$$
where $A(B)$ denotes the family of holomorphic mappings from $B$ to $A$ and $\D$ is the unit disc in $\C$.
\end{definition}

Let $\W=\set{r<0}\subset\subset\C^n$ be a smoothly bounded domain and $P\in\p\W$. Let $P_\dd=P-\dd\nu$, where $\nu$ is the unit outward normal vector to $\p\W$ at $P$, i.e., $\nu=\nabla r(P)/\norm{\nabla r(P)}$.

We want to estimate $F_K(P_\dd,\nu)$ as $\dd>0$ goes to $0$, i.e., as the point $P_\dd\in\W$ approaches the boundary point $P$ along the normal line to the boundary.

We see that the mapping $\phi(\z)=P_\dd+\z\dd\nu$ lies in $\W$ for all $\z\in\D$ for $\dd>0$ small enough. Hence we see that
$$
F_K(P_\dd,\nu)\le\frac{1}{\dd},
$$
for $\dd>0$ small enough.

The question is whether the other direction is true, i.e., whether there exists some constant $C>0$ such that
\begin{equation}\label{1011}
F_K(P_\dd,\nu)\ge C\frac{1}{\dd}
\end{equation}
for $\dd>0$ small enough.

The answer is ``No''. It is not true in general. Krantz \cite{Krantz 1992} showed that if $P$ is a strongly pseudoconcave point, then one has the following estimate:
$$
F_K(P_\dd,\nu)\approx\frac{1}{\dd^{3/4}}.
$$

Ian Graham \cite{Graham} proved that if $\W$ is a strongly pseudoconvex domain in $\C^n$, $n\ge 2$, then (\ref{1011}) holds. Catlin \cite{Catlin} showed (\ref{1011}) is true  when $\W$ is a finite type pseudoconvex domain in $\C^2$ and Lee \cite{Lina} proved it for convex domains in $\C^n$.

So it has been conjectured that  the estimate (\ref{1011}) must hold for a smoothly bounded pseudoconvex domain in $\C^n$, $n\ge 2$.

In this paper, we give a counter example to the conjecture. We construct a smoothly bounded infinite type pseudoconvex domain in $\C^3$ where we can find sequences $\dd_n\searrow 0$ and $a_n\nearrow\infty$ such that
\begin{equation}\label{1051}
F_K(P_{\dd_n},\nu)\le \frac{1}{a_n\dd_n}, \quad\text{for all $n\in\N$}.
\end{equation}

To construct such a domain, we modify the domain in Krantz \cite{Krantz93}, which is a smoothly bounded pseudoconvex domain in $\C^2$, where there exist sequences $b_n\nearrow\infty$, $c_n\nearrow\infty$ and $\dd_n\searrow 0$ such that
\begin{equation}\label{1053}
F_K(P_{\dd_n}, \nu+b_n T)\le\frac{1}{c_n\dd_n},\quad T\in T_P^\C(\p\W),\,\forall n\in\N.
\end{equation}

We see that the vector $\nu+b_nT$ has a very large tangential component for a large number  $n$. Therefore, as $P_{\dd_n}$ approaches the boundary point $P$, the estimate above gives the estimate for an almost tangential vector, not the normal vector.

In section 2, we give a detailed proof of (\ref{1053}) since the proof is very sketchy in \cite{Krantz93}. In section 3, we modify the example in section 2 and construct a smoothly bounded pseudoconvex domain in $\C^3$, where (\ref{1051}) holds.

\section{Construction of a domain in $\C^2$}

\begin{prop}
For a given increasing sequence $a_n\nearrow\infty$, we can construct a smoothly bounded pseudoconvex domain $\W=\set{r<0}\subset\C^2$ and find a boundary point $P\in\p\W$ and a sequence $\dd_n\searrow 0$ such that
$$
F_K(P_{\dd_n}, X_n)\le\frac{1}{a_n\dd_n},\quad P_{\dd_n}=P-\dd_n\nu,
$$
where $X_n$'s are vectors such that $(X_n,\nu)=1$ and $\nu$ is a unit outward normal vector to $\p\W$ at $P$.
\end{prop}
\begin{proof}
Let $\W\subset\C^2$ be a pseudoconvex domain defined as follows:
$$
\W=\set{(z,w)\in\C^2: r(z,w)=\re w+\rho(z)<0}\cap B(0,2),
$$
where $\rho(z)$ vanishes to infinite order at $0$ and satisfies the following:
\begin{equation}\label{943}
\rho(z)<\dd_n-a_n\frac{\dd_n}{r_n}\re z,\quad\forall\abs z\le r_n,\, r_n\searrow 0.
\end{equation}
Then the analytic disc
$$
\phi(\z)=\left(r_n \z, -\dd_n +a_n\dd_n \z\right),\quad \z\in\D
$$
lies inside $\W$ since, by (\ref{943}), we have
\begin{multline}
r(\phi(\z))=-\dd_n+a_n\dd_n\re \z+\rho(r_n \z)\\
<-\dd_n+a_n\dd_n\re \z+\dd_n -a_n\dd_n\re \z=0,\quad\forall \z\in\D.
\end{multline}
Hence, letting $P=0$, $P_{\dd_n}=P-\dd_n\nu=(0,-\dd_n)$ and
$$
X_n=\frac{r_n}{a_n\dd_n}\frac{\p}{\p z}+\frac{\p}{\p w},
$$
we have
$$
\phi'(0)=r_n\frac{\p}{\p z}+a_n\dd_n\frac{\p}{\p w}=a_n\dd_nX_n.
$$
Therefore we get
$$
F_K(P_{\dd_n}, X_n)\le \frac{1}{a_n\dd_n}.
$$
\noindent{\bf Construction of $\rho$.}

\medskip

Choose $r_n$'s such that
\begin{equation}\label{1124}
r_n=a_n^{-1}r_{n-1}^2,\quad r_n\le\frac{1}{4}
\end{equation}
and let
$$
u_n(z)=\frac{1}{8}-\re z+\frac{\log\abs z}{4\log a_n}.
$$
Define a function $R_n(z)$ as follows:
\begin{equation}\label{613}
R_n(z)=
\begin{cases}
\max\set{u_n(z), 0},& \re z\le   b_n\\
u_n(z), & \re z>b_n,
\end{cases}
\end{equation}
where $0< b_n\le 1$ is the smallest positive number such that
$$
\frac{1}{8}-b_n+\frac{\log b_n}{4\log a_n}=0.
$$
 \begin{figure}[h!]
    \centerline{\includegraphics[width=2.5in]{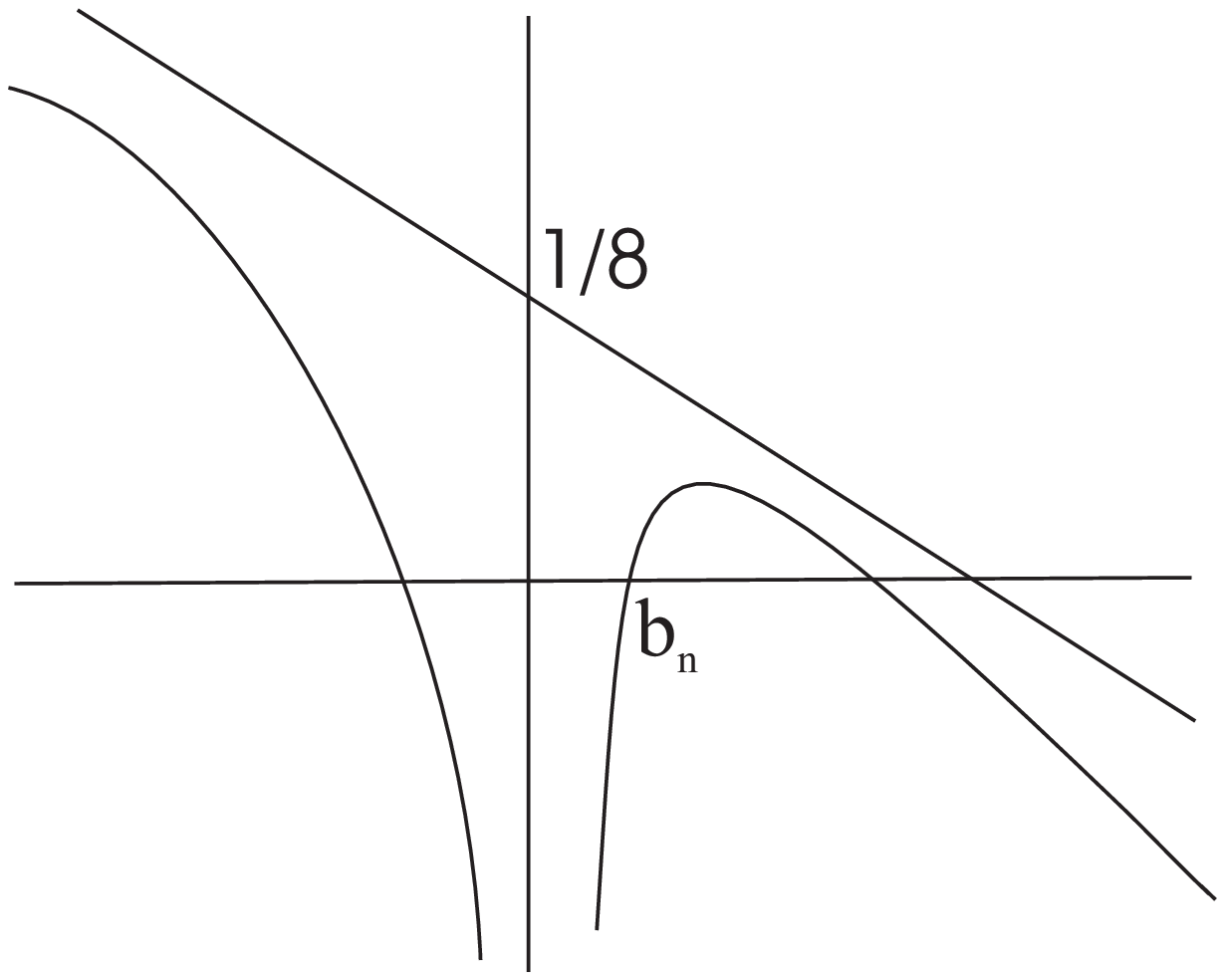}}
    \end{figure}
Such $b_n$ exists for large enough $a_n$ and we may assume $a_1$ is large enough. Note that $b_n\searrow 0$ and that $b_n\le 1/8$.

 We show that  $R_n(z)$ is subharmonic on $\C$. From (\ref{613}), we see that the function $R_n(z)$ is subharmonic on $\C\setminus\set{\re z=b_n}$. We can check that that $R_n$ is continuous near $\set{\re z=b_n}$. On the line $\re z=b_n$, we have that $\abs z\ge b_n$ and that
$$
\frac{1}{8}-\re z+\frac{\log\abs z}{4\log a_n}\ge \frac{1}{8}-b_n+\frac{\log b_n}{4\log a_n}=0,\quad \re z=b_n.
$$
Hence we conclude that $R_n$ is continuous and hence $R_n(z)$ is subharmonic on $\C$.

Calculating $u_n(z)$ for $\abs z<2r_n$, we can prove that $R_n(z)\equiv 0$ for all $\abs z<2r_n$. 

If $\abs z<2 r_n$, ($2r_n\le b_n$) then we get
\begin{multline}
u_n(z)=\frac{1}{8}-\re z+\frac{\log\abs z}{4\log a_n}\le \frac{1}{8}+2a_n^{-1}r_{n-1}^2+\frac{\log 2a_n^{-1}r_{n-1}^2}{4\log a_n}\\
\le\frac{1}{8}+2a_n^{-1}+\frac{-\log a_n+\log 2r_{n-1}^2}{4\log a_n}
\le -\frac{1}{8}+2a_n^{-1}<0.
\end{multline}
Therefore we check that
\begin{equation}\label{938}
R_n(z)\equiv 0,\quad\forall \abs z<2 r_n.
\end{equation}

Also we can easily check that $R_n(z)\le 3/8 -\re z$ for all $\abs z<a_n$. Calculate $u_n(z)$ for $\abs z<a_n$:
$$
u_n(z)=\frac{1}{8}-\re z+\frac{\log\abs z}{4\log a_n}\le \frac{3}{8}-\re z,\quad\forall\abs z<a_n
$$
Since $b_n$ is a small number, we have that 
$$
0\le \frac{3}{8}-\re z,\quad \re z\le b_n.
$$
Hence, from (\ref{613}),  we have
\begin{equation}\label{209}
R_n(z)\le \frac{3}{8}-\re z,\quad\forall \abs z< a_n.
\end{equation}

Now we make $R_n$ smooth using convolution with a smooth function. Choose a nonconstant $C^\infty$ function $\chi:\C\to\R^+\cup\set{0}$ such that
$$
0\le\chi\le 1,\quad \chi(z)=\chi(\abs z),\quad \chi(z)\equiv 0,\;\text{if }\abs z\ge 1
$$
and let
$$
m=\int_\C\chi(z) dxdy.
$$
Let $ d\mu=dx dy/m$.
We define the function $\tilde R_n(z)$ as follows
\begin{equation}\label{235}
\tilde R_n(z)=\int_\C R_n(z-\e_n w)\chi(w) d\mu(w),\quad \e_n< \frac{r_n}{2}.
\end{equation}
From (\ref{938}), we have that
$$
R_n(z-\e_n w)=0,\quad\forall \abs z<r_n, \;\abs w<1.
$$
Hence 
\begin{equation}\label{1000}
\tilde R_n(z)\equiv 0,\quad \abs z<r_n.
\end{equation}

Note that
\begin{equation}\label{947}
\abs{u_n(z-\e_n w)-u_n(z)}\le \e_n+\frac{\log\left(1+\frac{\e_n}{r_n}\right)}{4\log a_n},\quad\abs z\ge r_n.
\end{equation}
Therefore, for $\abs z\ge r_n$, from (\ref{947}) we have 
\begin{equation}\label{948}
\abs{R_n(z-\e_n w)-R_n(z)}\le \frac{1}{8}
\end{equation}
uniformly in $ \abs w\le 1$  and $ \abs z\ge r_n$ for small enough $\e_n$.
Together with 
$$
R_n(z)=\tilde R_n(z)=0,\quad\forall\abs z<r_n,
$$
we get 
\begin{equation}\label{957}
0\le \tilde R_n(z)-R_n(z)\le \frac{1}{8},\quad\forall z\in\C.
\end{equation}
The lower estimate $\tilde R_n(z)\ge R_n(z)$ holds by subharmonicity of $R_n$.

Therefore, from (\ref{209}) and (\ref{957}), we have
\begin{equation}\label{958}
\tilde R_n(z)\le\frac{1}{2}-\re z,\quad\forall\abs z<a_n.
\end{equation}

Let
$$
\rho_n(z)= \tilde R_n\left(\frac{a_nz}{r_n}\right).
$$
Then, from (\ref{958}),  we have
\begin{equation}\label{1032}
\rho_n(z)\le \frac{1}{2}-\frac{a_n}{r_n}\re z,\quad\forall \abs z<r_n
\end{equation}
and, from (\ref{957}), 
$$
\Abs{\tilde R_n\left(\frac{a_n z}{r_n}\right)-R_n\left(\frac{a_n z}{r_n}\right)}\le \frac{1}{8},\quad\forall z\in\C.
$$

By (\ref{1000}), we have 
\begin{equation}\label{1015}
\rho_n(z)=\tilde R_n\left(\frac{a_n z}{r_n}\right)\equiv 0,\quad\forall\abs z<r_n^2 a_n^{-1}.
\end{equation} 

 Also, from (\ref{1124}), we get
\begin{equation}\label{1016}
r_{n+1}=a_{n+1}^{-1}r_n^2< r_n^2 a_n^{-1}.
\end{equation}
Therefore, from (\ref{1015}) and (\ref{1016}), we have
\begin{equation}\label{1028}
\rho_n(z)\equiv 0,\quad\forall\abs z<r_{n+1}.
\end{equation}

We define $\rho(z)$ as follows:
\begin{equation}\label{1114}
\rho(z)=\sum_{k=1}^\infty \dd_k\rho_k(z),
\end{equation}
where $\dd_k$'s are positive numbers that will be chosen later. 
Since $r_n\searrow 0$, from (\ref{1028}), we see that
$$
\rho(z)=\sum_{k=n}^\infty \dd_k\rho_k(z),\quad\forall \abs z< r_n.
$$
 Therefore, from (\ref{1032}), we have
\begin{equation}\label{1055}
\rho(z)=\sum_{k=n}^\infty\dd_k \rho_k(z)\le \frac{\dd_n}{2}-a_n\frac{\dd_n}{r_n}\re z+\sum_{k=n+1}^\infty \dd_k\rho_k(z),\quad\abs z<r_n.
\end{equation}
Now we evaluate $\rho_k(z)$ for $\abs z< r_n$, $n<k$.
\begin{multline}
\rho_k(z)=\tilde R_k\left(\frac{a_k z}{r_k}\right)
\le R_k\left(\frac{a_k z}{r_k}\right)+\frac{1}{8}
\le\Abs{u_k\left(\frac{a_k z}{r_k}\right)}+\frac{1}{8}\\
\le\frac{1}{4}+\frac{a_k}{r_k}r_n+\frac{\log\abs{a_kr_n/r_k}}{4\log a_k}
\le\frac{1}{4}+\frac{a_k}{r_k}+\frac{\log a_k+\log (r_n/r_k)}{4\log a_k}\\
\le\frac{1}{2}+\frac{a_k}{r_k}+\frac{\log (1/r_k)}{4\log a_k},\quad\abs z<r_n
\end{multline}
Let
$$
A_k=\frac{1}{2}+\frac{a_k}{r_k}+\frac{\log (1/r_k)}{4\log a_k}
$$
and choose $\dd_k$'s as follows:
$$
\dd_{k}\le\frac{\dd_{k-1}}{A_k}\frac{1}{2^{k}},\quad\dd_k<1.
$$
Then we see that $\dd_k\searrow 0$ and
$$
\dd_k\rho_k(z)\le \dd_k A_k\le\dd_{k-1}\frac{1}{2^k}\le \dd_n\frac{1}{2^k},\quad \abs z<r_n, \;k>n.
$$
Hence we have
\begin{equation}\label{1056}
\sum_{k=n+1}^\infty \dd_k\rho_k(z)\le \dd_n\sum_{k=n+1}^\infty \frac{1}{2^k}\le\frac{\dd_n}{2},\quad \abs z<r_n.
\end{equation}
Therefore, from (\ref{1055}) and (\ref{1056}),  we get
\begin{equation}\label{1128}
\rho(z)<\dd_n-a_n\frac{\dd_n}{r_n}\re z,\quad\forall \abs z<r_n.
\end{equation}
Note that we can choose $\dd_k$'s small enough that $\rho$ is smooth.
\end{proof}

\section{Construction of a domain in $\C^3$}

\begin{theorem}
For a given increasing sequence $a_n\nearrow\infty$, we can construct a smoothly bounded infinite type pseudoconvex domain $\W\subset\subset\C^3$  and find a sequence $\dd_n\searrow 0$ such that with a suitable point $P\in\p\W$, one has
$$
F_K(P_{\dd_n},\nu)\le\frac{1}{a_n\dd_n},\quad P_{\dd_n}=P-\dd_n\nu,
$$
where $\nu$ is the unit outward normal vector to $\p\W$ at $P\in\p\W$.
\end{theorem}
\begin{proof}
Let $\W\subset\C^3$ be a pseudoconvex domain defined as follows:
$$
\W=\set{(s,t,w)\in\C^3: r(s,t,w)=\re w+\tilde\rho(s,t)<0}\cap B(0,2).
$$

We construct  $\tilde\rho(s,t)$ such that for a given sequence $a_n\nearrow\infty$, $a_n\ge 4$,  we can find a sequence $\dd_n\searrow 0$ and $r_n\searrow 0$ such that  $\tilde\rho(s,t)$ satisfies
$$
\tilde\rho(\z^3,\z^2)<\dd_n-a_n\frac{\dd_n}{r_n}\re \z,\quad\forall\abs \z\le r_n.
$$
Then the analytic disc
$$
\phi(\z)=(r_n^3 \z^3, r_n^2 \z^2, -\dd_n +a_n\dd_n \z),\quad \z\in\D
$$
lies inside $\W$ since
\begin{multline}
r(\phi(\z))=-\dd_n+a_n\dd_n\re \z+\tilde\rho(r_n^3 \z^3,r_n^2\z^2)\\
<-\dd_n+a_n\dd_n\re \z+\dd_n -a_n\dd_n\re \z=0,\quad\forall \z\in\D.
\end{multline}
Hence, letting $P=0$ and $P_{\dd_n}=P-\dd_n\nu=(0,-\dd_n)$, we have
$$
\phi(0)=P_{\dd_n},\quad\phi'(0)= a_n\dd_n\frac{\p}{\p w}.
$$
Therefore we get
$$
F_K(P_{\dd_n}, \nu)\le \frac{1}{a_n\dd_n}.
$$
\noindent{\bf Construction of $\tilde\rho$}
\medskip

We modify the function $\rho$ we defined in section 2: Refer (\ref{1114}), (\ref{1028}), and (\ref{1128}).

Here we restate  the definition and the properties of $\rho$:
$$
\rho(\z)=\sum_{j=1}^\infty \dd_j\rho_j(\z),
$$
where $\rho_j(\z)$ satisfies
\begin{gather}
\rho_j(\z)\equiv 0,\quad \abs \z<r_{j+1}\\
\rho_j(\z)\le \frac{1}{2}-\frac{a_j}{r_j}\re \z,\quad\forall \abs \z<r_j.
\end{gather}

Let $z=(s,t)\in\C^2$ and $V=\set{s^2-t^3=0}\subset\C^2$. We define 
$$
\tilde\rho_n(s,t)=\rho_n(s/t)=\rho_n(\z),\quad \text{if } (s,t)=(\z^3,\z^2)\in V.
$$
 Now we extend $\tilde\rho_n$ to $\C^2$.

Let $\tilde r_n=r_{n+1}^3$ and $B_n=B(0,\tilde r_n)\subset\C^2$. If $(s,t)\in B_n\cap V$, then $(s,t)=(\z^3,\z^2)$ and $\abs{\z^3}\le r_{n+1}^3$ and $\abs{\z^2}\le r_{n+1}^3$. Hence $\abs\z\le r_{n+1}$ ($r_{n+1}<1$). So we know that
$$
\tilde\rho_n(s,t)=0,\quad \forall(s,t)\in V\cap B_n.
$$
We let
$$
\tilde\rho_n(s,t)\equiv 0,\quad\forall (s,t)\in B_n.
$$
Let $B_n'=B(0,3\tilde r_n/4)\subset B_n$ and choose a small neighborhood $U_n$ of $V$ such that the projection $\pi:U_n\To V$ is well defined on $U_n\setminus B_n'$ and that  $U_n\setminus B_n'=\set{p\in\C^2:\abs{p-\pi(p)}<d_n}$ for suitable small positive numbers $d_n$.

 \begin{figure}[h!]
    \centerline{\includegraphics[width=2.5in]{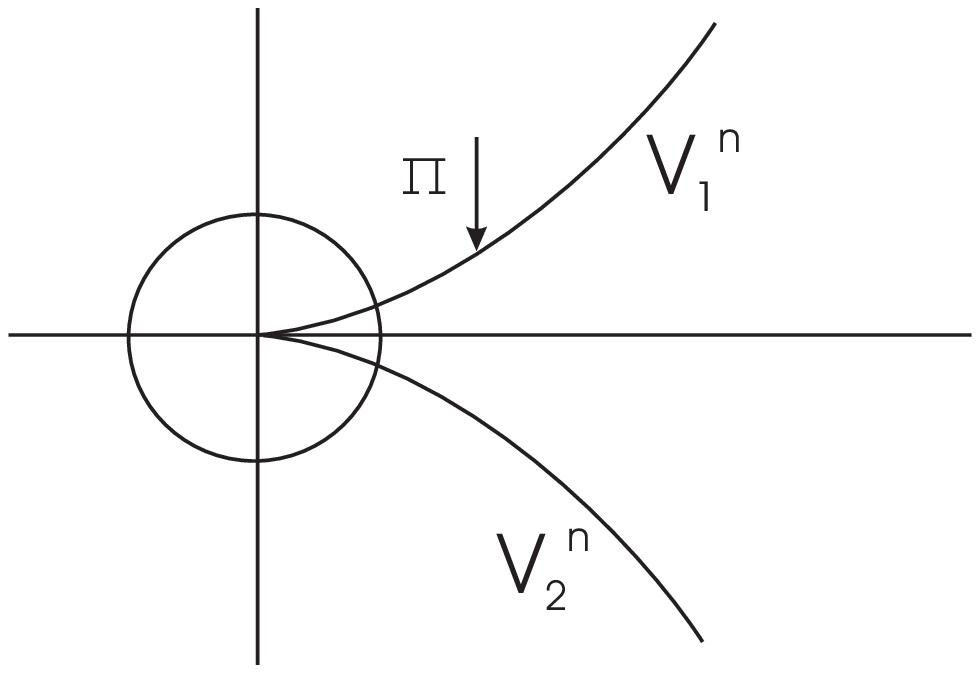}}
    \end{figure}
We define the projection $\pi$ more precisely. Let $V_1$ and $V_2$ be the two sheets of $V$: $V_1=\set{(r^{3/2}e^{i(3\theta/2+\pi)},re^{i\theta}): r,\theta\in\R}$, $V_2=\set{(r^{3/2}e^{i(3\theta/2)},re^{i\theta}):r,\theta\in\R}$ and let $V_1^n=V_1\setminus B_n'$ and $V_2^n=V_2\setminus B_n'$.  We consider a biholomorphic mapping
$$
\Phi:(s,t)\To (s, t-s^{2/3})
$$
in a small neighborhood $U_n$  of $V_1^n\cup V_2^n$. Then $\Phi(z)=(s,0)$ if $z\in V_1^n\cup V_2^n$. We define the projection $\pi$  as follows:
$$
\pi(z)=\Phi^{-1}(\pi_1(\Phi(z))),
$$
where $\pi_1(s,t)=(t,0)$.

We define $\tilde\rho_n$ on $U_n\sm B_n$ as follows:
$$
\tilde\rho_n(z)=\rho_n(\pi(z)),\quad z\in U_n\sm B_n.
$$

Then the function $\tilde\rho_n$ is well defined on $B_n\cup U_n$.
Now we extend $\tilde\rho_n$  to $\C^2$
Choose a smooth function $h:\R\To[0,1]$ such that
$$
h(x)=
\begin{cases}
0,& x\in[0,(3/4)^2]\\
1,& x\ge 1
\end{cases}
$$
Let $\chi:\R\To[0,1]$ be a smooth function defined as follows:
$$
\chi(x)=
\begin{cases}
1,& x\in[0,1/2]\\
0,& x\ge 1
\end{cases}
$$
and let 
$$
\chi_n(z)=\chi\left(h\left(\frac{\abs z^2}{\tilde r_n^2}\right)\frac{\abs{z-\pi(z)}^2}{d_n^2}\right).
$$
Then $\chi_n$ satisfies
$$
\chi_n(z)=
\begin{cases}
\chi\left(\frac{\abs{z-\pi(z)}^2}{d_n^2}\right),& z\in \C^2\sm B_n\\
\chi\left(h\left(\frac{\abs z^2}{\tilde r_n^2}\right)\frac{\abs{z-\pi(z)}^2}{d_n^2}\right),& z\in  B_n\sm B_n'\\
1,& z\in B_n'
\end{cases}
$$
Consider the function $\tilde\rho_n\chi_n$. Then $\tilde\rho_n\chi_n$ is a smooth extension of $\tilde\rho_n$ on $\C^2$ and it satisfies the following:
$$
\tilde\rho_n\chi_n(z)=
\begin{cases}
0,& z\in B_n\\
\tilde\rho_n,& z\in U\sm B_n,\; \abs{z-\pi(z)}^2\le \frac{d_n^2}{2}\\
\tilde\rho_n\chi\left(\frac{\abs{z-\pi(z)}^2}{d_n^2}\right), &z\in U\sm B_n,\;\frac{d_n^2}{2}\le\abs{z-\pi(z)}^2\le d_n^2\\
0,& z\not\in U_n\cup B_n
\end{cases}
$$

Let
$$
p_n(z)=\tilde\rho_n\chi_n(z).
$$
We can find $C_n\ge 0$ such that
$$
\p\cl\p p_n(z)(L,\cl L)\ge -C_n\norm L^2,\quad \forall z\in B(0,2).
$$
Now we add a strictly plurisubharmonic function to make it plurisubharmonic.
Note that $p_n$ is plurisubharmonic everywhere except on 

$$
A_n=\set{z\in U_n\sm B_n:\frac{d_n^2}{2}\le\abs{z-\pi(z)}^2\le d_n^2}.
$$
Let
$$
q(z)= e^{\abs z^2}\abs{s^2-t^3}^2,\quad z=(s,t).
$$

Then $q(z)$ is strictly plurisubharmonic outside a small neighborhood of $V$ since $\log q$ is strictly plurisubharmonic.
Hence  we can find a number $c_n>0$ such that
$$
\p\cl\p q(z)(L,\cl L)\ge c_n\norm L^2,\quad\forall z\in A_n\cap B(0,2)
$$
Choose $K_n>0$ large enough that
$$
-C_n+K_n c_n\ge 0
$$
Then $p_n+K_nq$ is plurisubharmonic on $A_n\cup D_n$ and hence in $\C^2$. Let
$$
r_n=p_n+K_n q.
$$
Then
$$
r_n(\z^3,\z^2)=p_n(\z^3,\z^2)=\tilde\rho_n(\z)=\rho_n(\z).
$$
Hence
$$
r_n(\z^3,\z^2)=
\begin{cases}
0,&\abs \z<r_{n+1}\\
\le\frac{1}{2}-\frac{a_n}{r_n}\re \z,&\abs\z<r_n.
\end{cases}
$$

Let
$$
\tilde\rho(z)=\sum_{j} \dd_j r_j,
$$
and choose $\dd_j$'s small enough that $\tilde\rho$ is smooth in a small neighborhood of $0$.
\end{proof}

\smallskip

\noindent John Erik Fornaess\\
Mathematics Department\\
The University of Michigan\\
East Hall, Ann Arbor, MI 48109\\
USA\\
fornaess@umich.edu\\

\smallskip
\noindent Lina Lee\\
Mathematics Department\\
The University of Michigan\\
East Hall, Ann Arbor, MI 48109\\
USA\\
linalee@umich.edu\\
\end{document}